\documentclass[12pt]{amsart}

\usepackage{amsmath,amssymb,amsthm,verbatim,amscd}
\usepackage{bm}
\usepackage[dvipdfmx]{graphicx}
\usepackage{color}
\usepackage[all]{xy}
\usepackage{braket}
\usepackage{enumerate}
\usepackage{mediabb}
\usepackage{overpic} 

\newtheorem{theorem}{Theorem}[section]

\newtheorem{prob}[theorem]{Problem}
\newtheorem{lem}[theorem]{Lemma}
\newtheorem{conj}[theorem]{Conjecture}
\newtheorem{ques}[theorem]{Question}

\theoremstyle{definition}
\newtheorem{ex}[theorem]{Example}
\newtheorem{definition}[theorem]{Definition}
\theoremstyle{remark}
\newtheorem{rem}[theorem]{Remark}

\newcommand{\N}{\mathbb{N}}
\newcommand{\Z}{\mathbb{Z}}
\newcommand{\R}{\mathbb{R}}

\newcommand{\sign}{\mathop{\mathrm{sign}}\nolimits}

\numberwithin{equation}{section}

\title{Annulus twist and diffeomorphic 4-manifolds II}

\author{Tetsuya Abe} 
\address{Department of Mathematics,
Tokyo Institute of Technology,
2-12-1 Ookayama, Meguro-ku, 
Tokyo 152-8551, Japan}
\email{abe.t.av@m.titech.ac.jp}

\author{In Dae Jong}
\address{
Department of Mathematics, 
Kinki University, 
3-4-1 Kowakae, Higashiosaka City,
Osaka 577-0818, Japan}
\email{jong@math.kindai.ac.jp}

\subjclass[2010]{57M25, 57R65}
\keywords{Kirby calculus; Annulus twist; Dehn surgery; 2-handle addition}

\begin{document}

\maketitle

\begin{abstract}
We solve a strong version of Problem 3.6 (D) in Kirby's list, that is, 
we show that for any integer $n$, there exist infinitely many mutually distinct knots 
such that $2$-handle additions along them with framing $n$ yield the same $4$-manifold. 
\end{abstract}

\section{Introduction}\label{sec:intro} 
For a knot $K$ in the $3$-sphere $S^3 = \partial B^4$, 
we denote by $M_K(n)$ the 3-manifold obtained from $S^3$ by $n$-surgery on $K$, 
and by $X_{K}(n)$ the smooth $4$-manifold obtained from 
$B^4$ by attaching a $2$-handle along $K$ with framing $n$. 
The symbol $\approx$ stands for a diffeomorphism.
In \cite{AJOT}, 
the authors, Omae, and Takeuchi asked the following problem,
a strong version of Problem 3.6 (D) in Kirby's list~\cite{Kirby} (see also Problem 2 in \cite{LO}). 

\begin{prob}\label{prob:Kirby}
Let $\gamma$ be an integer. 
Find infinitely many mutually distinct knots $K_1, K_2, \dots$ 
such that $X_{K_i}(\gamma) \approx X_{K_j}(\gamma) $ for each $i, j \in \N$. 
\end{prob}

In \cite{Ak1, Ak2},
Akbulut gave a partial answer to Problem~\ref{prob:Kirby} 
by finding a pair of distinct knots $K$ and $K'$ 
such that $X_{K}(\gamma) \approx X_{K'}(\gamma)$ for each $\gamma \in \Z$. 
Using an annulus twist introduced by Osoinach~\cite{Osoinach}, 
Problem~\ref{prob:Kirby} was solved affirmatively for $\gamma=0, \pm4$ in \cite{AJOT}. 

In this paper, we generalize an annulus twist in a somewhat unexpected way, 
and solve Problem~\ref{prob:Kirby} affirmatively by using the new operation. 

\begin{theorem}\label{thm:main}
For every $n \in \Z$, 
there exist distinct knots $J_0, J_1, J_2, \dots$ such that
\[ X_{J_0}(n) \approx X_{J_1}(n) \approx X_{J_2}(n) \approx \cdots \, . \]
\end{theorem}

The knots $J_0$ and $J_1$ 
in Theorem~\ref{thm:main} (for $n>0$) are depicted in Figure~\ref{fig:8_20-1}.
In the figure, the rectangle labelled $n$ stands for $n$ times right-handed full twists.
Note that $J_0$ is the knot $8_{20}$ in Rolfsen's table~\cite{Rolfsen}. 

\begin{figure}[!htb]
\centering
\begin{overpic}[width=.65\textwidth]{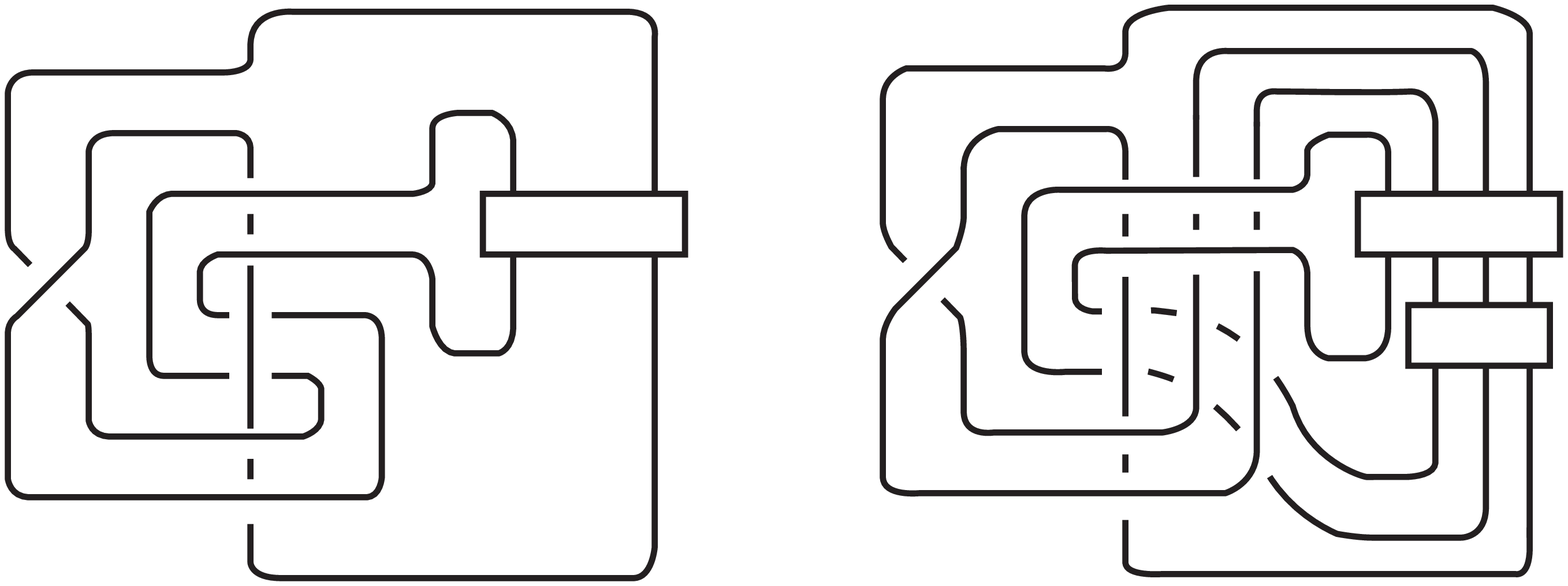}
\put(25,-4){$J_0$}
\put(82,-4){$J_1$}
\put(36.5,21.5){$1$}
\put(92.5,21.5){$1$}
\put(93.5,14.8){$n$}
\end{overpic}
\caption{The knots $J_0$ and $J_1$ such that $X_{J_0}(n) \approx X_{J_1}(n)$. }
\label{fig:8_20-1}
\end{figure}

This paper is organized as follows: 
In Section~\ref{sec:construction}, 
we recall the definition of an annulus presentation of a knot 
and introduce the notion of a ``simple'' annulus presentation. 
We define a new operation $(*n)$ on an annulus presentation, 
which is a generalization of an annulus twist. 
For a knot $K$ with an annulus presentation and an integer $n$, 
we construct a knot $K'$ (with an annulus presentation) such that 
$M_{K}(n) \approx M_{K'}(n)$ by using the operation $(*n)$ (Theorem~\ref{thm:diffeo3}). 
In Section~\ref{sec:extension}, 
for a knot $K$ with a simple annulus presentation and any integer $n$, 
we construct a knot $K'$ (with a simple annulus presentation) such that 
$X_{K}(n) \approx X_{K'}(n)$ by using the operation $(*n)$ (Theorem~\ref{thm:diffeo4}). 
Note that the two knots $K$ and $K'$ are possibly the same. 
In Section~\ref{sec:proof}, 
we introduce the notion of a ``good'' annulus presentation, and 
show that, for a given knot with a good annulus presentation, 
the infinitely many knots constructed by using the operation $(*n)$ 
have mutually distinct Alexander polynomials when $n \ne 0$ (Theorem~\ref{thm:main2}). 
This yields Theorem~\ref{thm:main} as an immediate corollary. 
In Appendix~\ref{sec:Cabling-conjecture}, 
we give a potential application of Theorem~\ref{thm:diffeo3} 
to the cabling conjecture.

\subsection*{Acknowledgments}
The authors would like to express their gratitude to 
Yuichi Yamada and other participants of 
handle seminar organized by Motoo Tange.
This paper would not be produced 
without Yamada's interest to annulus twists. 
The first author was supported by JSPS KAKENHI Grant Number 25005998.

\section{Construction of knots}\label{sec:construction}

\subsection{Annulus presentation}\label{subsec:AP} 
We recall the definition of an annulus presentation\footnote{In \cite{AJOT},
it was called a \emph{band presentation}.} of a knot. 
Let $A \subset \R^2 \cup \{ \infty \} \subset S^3$ be a trivially embedded annulus
with an $\varepsilon$-framed unknot $c$ in $S^3$ 
as shown in the left side of Figure~\ref{fig:Def-AP}, 
where $\varepsilon = \pm 1$. 
Take an embedding of a band $b: I \times I \to S^3$ such that 
\begin{itemize}
\item $b(I \times I) \cap \partial A = b(\partial I \times I)$, 
\item $b(I \times I) \cap \text{int} A$ consists of ribbon singularities, and
\item $b(I \times I)  \cap c= \emptyset$,
\end{itemize}
where $I = [0,1]$. 
Throughout this paper, we assume that $A \cup b(I \times I)$ is orientable. 
This assumption implies that the induced framing is zero (see \cite{AJOT}). 
Unless otherwise stated, we also assume for simplicity that $\varepsilon=-1$. 
If a knot $K$ in $S^3$ is isotopic to the knot 
$\left( \partial A \setminus b(\partial I \times I)\right) \cup b( I \times \partial I)$ 
in $M_{c}(-1) \approx S^3$, 
then we say that $K$ admits an \emph{annulus presentation} $(A,b,c)$.
It is easy to see that a knot admitting an annulus presentation 
is obtained from the Hopf link by a single band surgery (see~\cite{AJOT}). 
A typical example of a knot admitting an annulus presentation is given 
in Figure~\ref{fig:Def-AP}. 

\begin{figure}[!htb]
\centering
\begin{overpic}[width=1.0\textwidth]{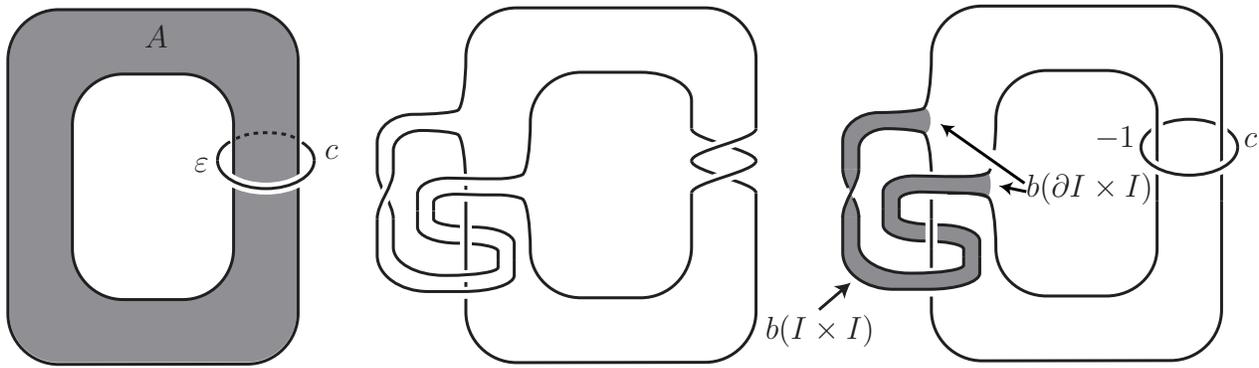}
\put(11.5,26){$A$}
\put(15.5,16){$\varepsilon$}
\put(26,17){$c$}
\put(61.5,2.5){$b(I \times I)$}
\put(82.5,14){$b(\partial I \times I)$}
\put(88,18){$-1$}
\put(100,18){$c$}
\end{overpic}
\caption{The knot depicted in the center admits an annulus presentation as in the right side.}
\label{fig:Def-AP}
\end{figure}

For an annulus presentation $(A,b,c)$, 
$\left( \R^2 \cup \{\infty\} \right) \setminus \textrm{int}A$ consists of 
two disks $D$ and $D'$, see Figure~\ref{fig:simple}. 
Assume that $\infty \in D'$. 

\begin{definition}\label{def:simple} 
An annulus presentation $(A,b,c)$ is called \textit{simple} if 
$b(I \times I) \cap \textrm{int} D = \emptyset$. 
\end{definition}

For example, in Figure~\ref{fig:simple}, 
the annulus presentation depicted in the center is simple, 
and the right one is not.

\begin{figure}[!htb]
\centering
\begin{overpic}[width=1.0\textwidth]{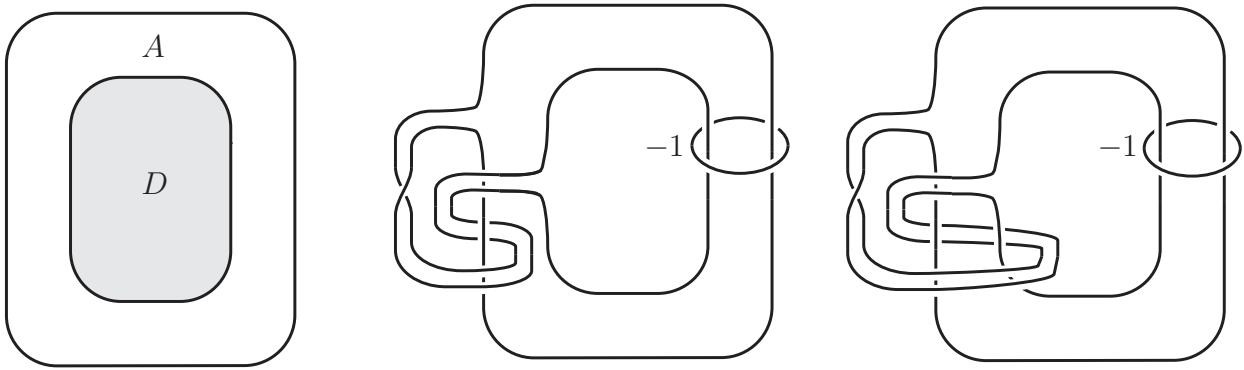}
\put(11,14){$D$}
\put(11,25){$A$}
\put(51.5,17){$-1$}
\put(88,17){$-1$}
\end{overpic}
\caption{The position of $D$, a simple annulus presentation and a non-simple annulus presentation.}
\label{fig:simple}
\end{figure}

Let $(A,b,c)$ be an annulus presentation of a knot. 
In a situation where it is inessential how the band $b(I \times I)$ is embedded, 
we often indicate $(A,b,c)$ in an abbreviated form as in Figure~\ref{fig:AP-abb}.

\begin{figure}[!htb]
\centering
\begin{overpic}[width=.25\textwidth]{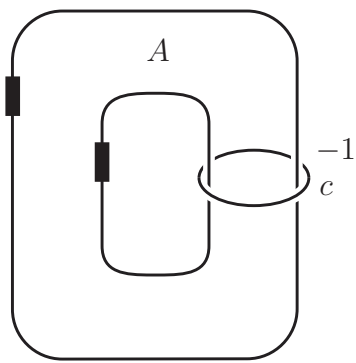}
\put(40,85){$A$}
\put(88,57){$-1$}
\put(89,47){$c$}
\end{overpic}
\caption{
Thick arcs stand for $b (\partial I \times I)$. }
\label{fig:AP-abb}
\end{figure}

\subsection{Operations}\label{subsec:operation}

To construct knots yielding the same $4$-manifold by a $2$-handle attaching, 
we define operations on an annulus presentation. 

\begin{definition}
Let $(A, b, c)$ be an annulus presentation, and $n$ an integer. 
\begin{itemize}
\item 
\emph{The operation $(A)$} is to apply an annulus twist\footnote{For the definition of an annulus twist, see~\cite[Section 2]{AT}.} along the annulus $A$. 
\item 
\emph{The operation $(T_n)$} is defined as follows: 
\begin{enumerate}
\item 
Adding the $(-1/n)$-framed unknot as in Figure~\ref{fig:Tn}, and 
\item 
(after isotopy) blowing down along the $(-1/n)$-framed unknot. 
\end{enumerate}
\item 
\emph{The operation $(*n)$} is the composition of $(A)$ and $(T_n)$. 
\end{itemize}
\end{definition}

In the operation $(T_n)$, 
the added $(-1/n)$-framed unknot is lying on the neighborhood of 
$c$ and $\partial A$, and does not intersect $b(I \times I)$. 
The intersection of $A$ and the added unknot is just one point. 

The operation $(*n)$ is a generalization of an annulus twist, in particular, $(*0) = (A)$.

\begin{figure}[!htb]
\centering
\begin{overpic}[width=.6\textwidth]{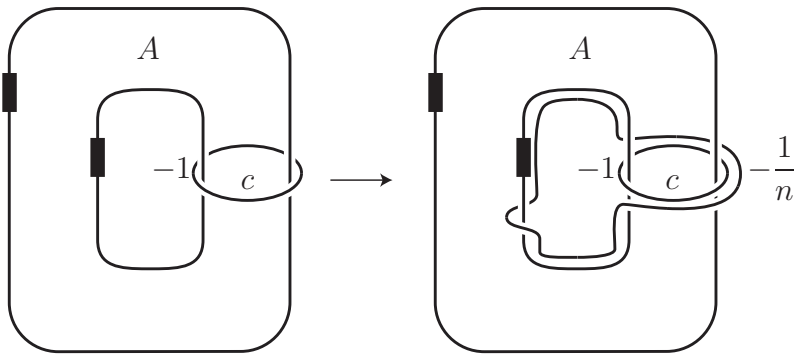}
\put(20,24){$-1$}
\put(77,24){$-1$}
\put(100,24){$-\dfrac1n$}
\put(18,40){$A$}
\put(76,40){$A$}
\put(32,22.5){$c$}
\put(89,22.5){$c$}
\end{overpic}
\caption{Add the $(-1/n)$-framed unknot in the operation $(T_n)$.}
\label{fig:Tn}
\end{figure}

\subsection{Construction}\label{subsec:Construction}

For a given knot $K$ with an annulus presentation, 
we can obtain a new knot $K'$ with a new annulus presentation 
by applying the operation $(*n)$. 
By abuse of notation, 
we call  $K'$ \emph{the knot obtained from $K$ by the operation $(*n)$}. 
Here we give examples. 

\begin{ex}\label{ex:1}
Let $J_{0}$ be the knot with the simple annulus presentation as in Figure~\ref{fig:*n}. 
Let $J_1$ be the knot obtained from $J_{0}$ by the operation $(*{n})$. 
Then $J_1$ is as in Figure~\ref{fig:*n}. 
\end{ex}
\begin{figure}[!htb]
\centering
\begin{overpic}[width=1.0\textwidth]{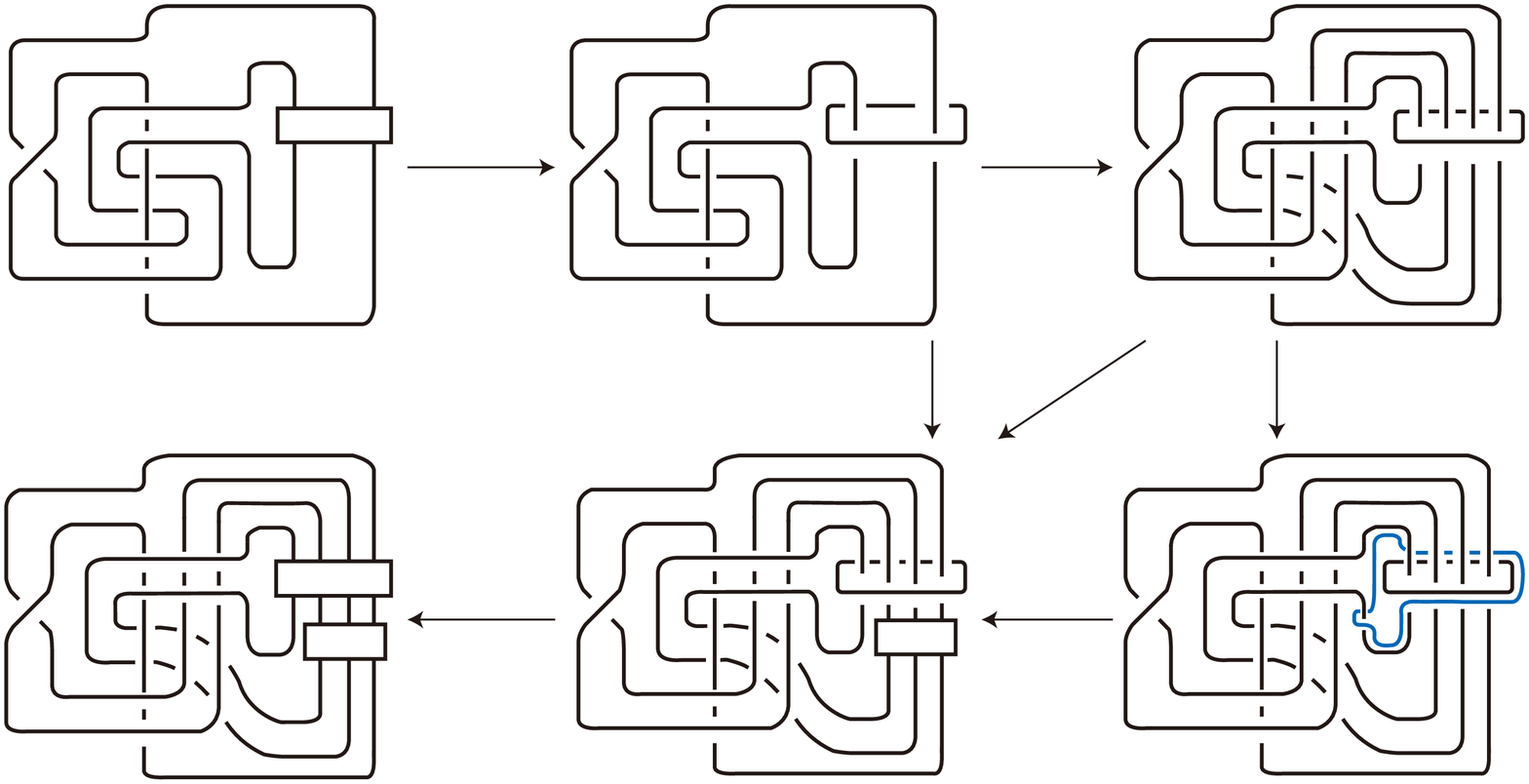}
\put(21.2,42){$1$}
\put(37,27.5){annulus presentation}
\put(37,-2){annulus presentation}
\put(26.5,41){blow up}
\put(62,44.5){$-1$}
\put(99,44.3){$-1$}
\put(62.3,14.5){$-1$}
\put(59.4,8.7){$n$}
\put(21.5,12.35){$1$}
\put(21.8,8.4){$n$}
\put(67,41){$(A)$}
\put(66,14.5){blow}
\put(66,12){down}
\put(29,14.5){blow}
\put(29,12){down}
\put(69.5,23){$(T_n)$}
\put(61.5,25.5){$(*n)$}
\put(11.5,27.5){$J_0$}
\put(11.5,-2.5){$J_1$}
\put(98.5,10){$\textcolor[cmyk]{1,0.5,0,0}{-\dfrac1n}$}
\end{overpic}
\vskip .3cm
\caption{By the operation $(*n)$, 
the knot $J_0$ with the annulus presentation is deformed into 
the knot $J_1$ with the annulus presentation.} 
\label{fig:*n}
\end{figure}

\begin{rem}
Let $K$ be a knot with an annulus presentation $(A, b, c)$, 
and $K'$ the knot obtained from $K$ by $(*n)$. 
If $(A, b, c)$ is simple, 
then the resulting annulus presentation of $K'$ is also simple. 
\end{rem}

\begin{ex}\label{ex:2}
For the knot $J_1$ in Example~\ref{ex:1} with $n = 1$, 
let $J_2$ be the knot obtained from $J_{1}$ by applying the operation $(*{1})$. 
Then $J_2$ is as in Figure~\ref{fig:ExK2}.
\end{ex}

\begin{figure}[!htb]
\centering
\begin{overpic}[width=\textwidth]{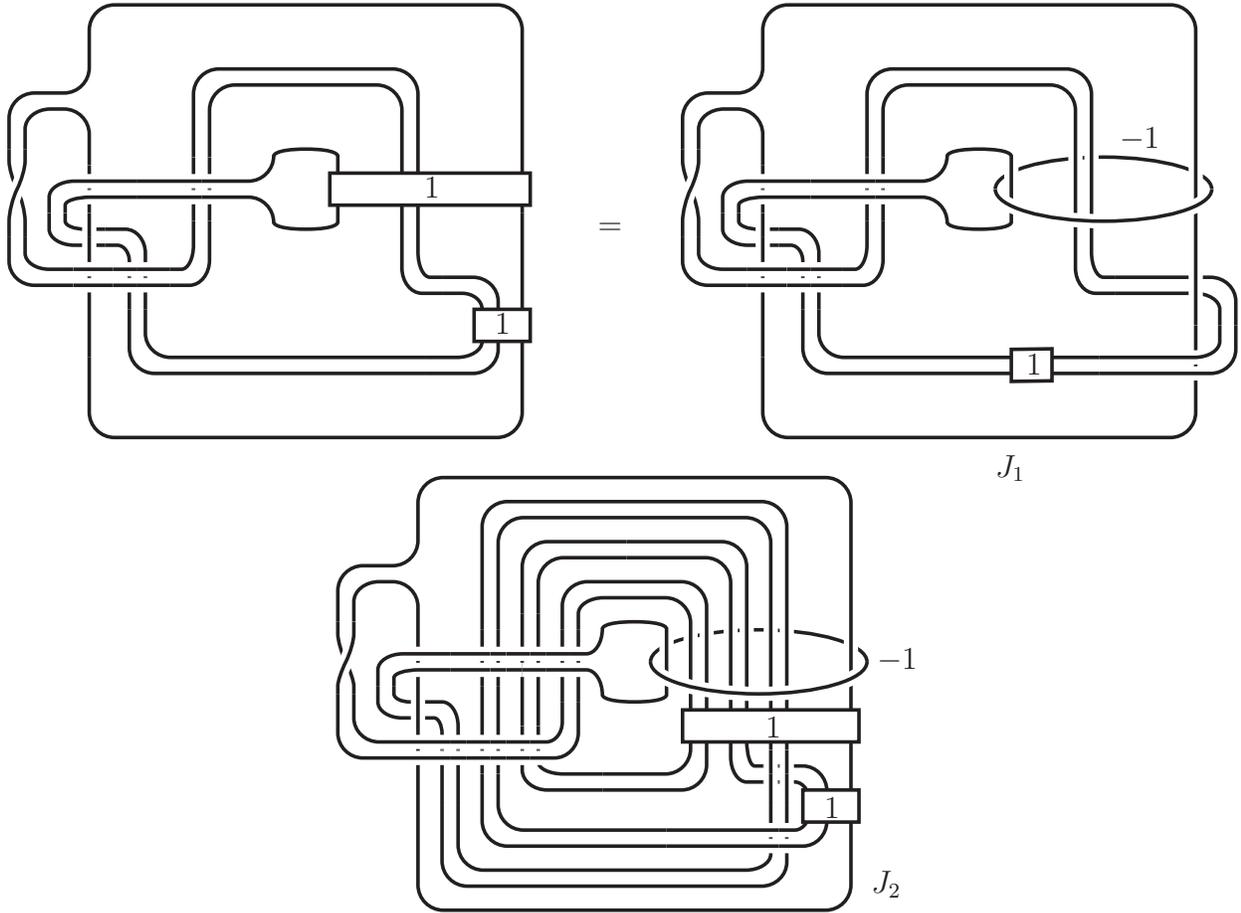}
\put(34,58){$1$}
\put(39.7,47){$1$}
\put(82.5,43.7){$1$}
\put(90,62){$-1$}
\put(70.5,20){$-1$}
\put(61.5,14.5){$1$}
\put(66.2,8){$1$}
\put(48,55){$=$}
\put(80,35.5){$J_1$}
\put(70,2){$J_2$}
\end{overpic}
\caption{An annulus presentation of the knot $J_2$ (lower half) obtained from $J_0$ 
by applying $(*1)$ two times. }
\label{fig:ExK2}
\end{figure}

The following lemma is obvious, however, important in our argument. 

\begin{lem}
Let $L$ be a $2$-component framed link 
which consists of $L_1$ with framing $(-1/n)$ and 
$L_2$ with framing $0$ as in the left side of Figure~\ref{fig:blow-down}. 
Suppose that the linking number of $L_1$ and $L_2$ is $\pm 1$ 
(with some orientation). 
Then two Kirby diagrams in Figure~\ref{fig:blow-down} 
represent  the same $3$-manifold. 
\end{lem}

\begin{figure}[!htb]
\centering
\begin{overpic}[width=.4\textwidth]{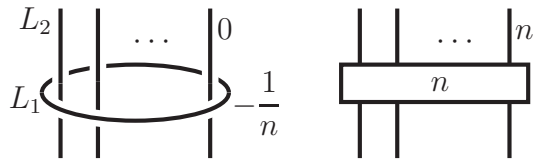}
\put(80,13.5){$n$}
\put(-5,10){$L_1$}
\put(-3,26){$L_2$}
\put(20,22){$\cdots$}
\put(81,22){$\cdots$}
\put(37,24){$0$}
\put(40,9){$-\dfrac1n$}
\put(97,24){$n$}
\end{overpic}
\caption{Two  Kirby diagrams represent  the same $3$-manifold. }
\label{fig:blow-down}
\end{figure}

\begin{theorem} \label{thm:diffeo3}
Let $K$ be a knot with an annulus presentation 
and $K'$ be the knot obtain from $K$ by the operation $(*n)$.
Then 
\[ M_{K}(n) \approx M_{K'}(n) \, . \] 
\end{theorem}
\begin{proof}
First, 
we consider the case where $K = J_0 = 8_{20}$ with the usual annulus presentation 
as in Figure~\ref{fig:*n}. 
Figure~\ref{fig:HM1} shows that 
$M_{K}(n) $ is represented by the last diagram in Figure~\ref{fig:HM1}, 
and this is diffeomorphic to $M_{K'}(n)$ by Figure~\ref{fig:HM2}. 
The moves in Figure~\ref{fig:HM2} correspond to the operation $(*n)$.

\begin{figure}[!htb]
\centering
\begin{overpic}[width=\textwidth]{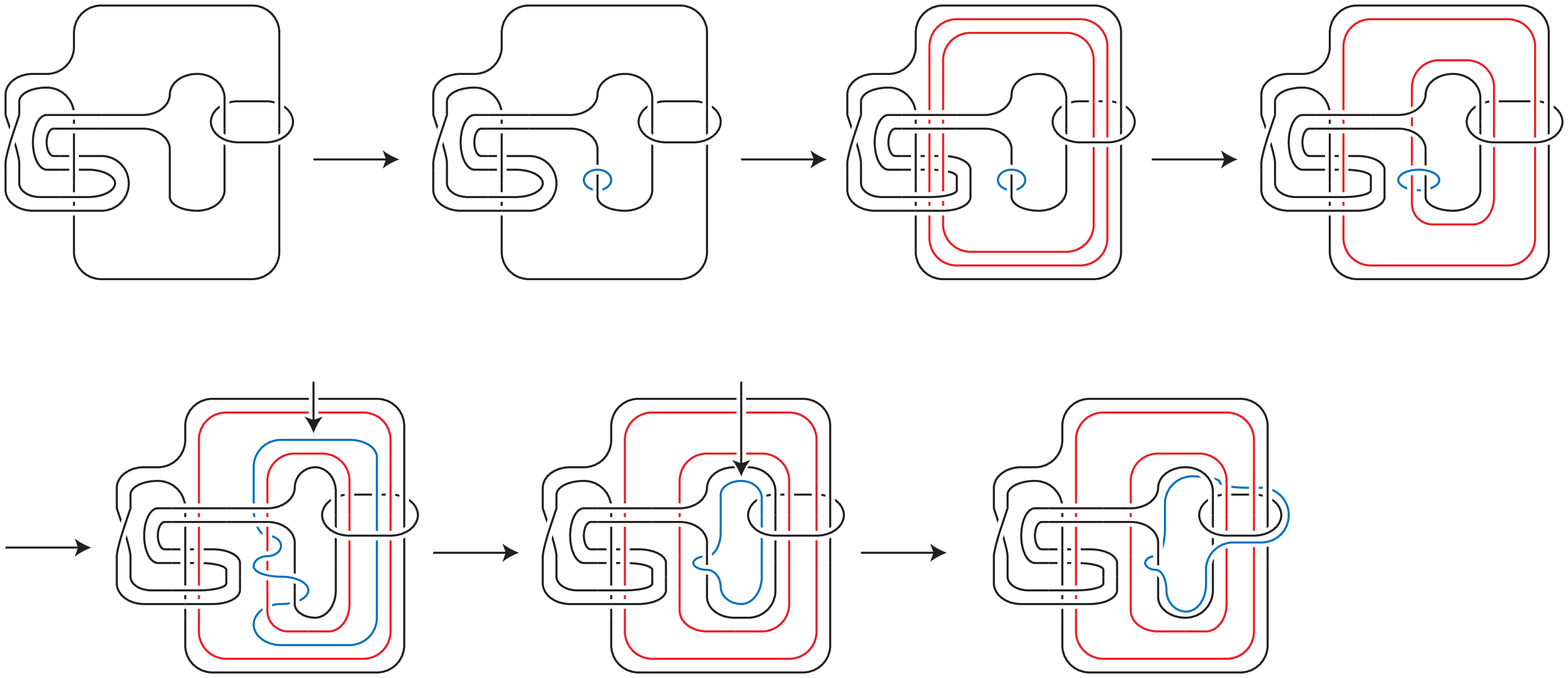}
\put(18.5,36){$-1$}
\put(46,36){$-1$}
\put(72.2,36){$-1$}
\put(99.5,36){$-1$}
\put(26.5,10.7){$-1$}
\put(53.7,10.7){$-1$}
\put(74.4,8){{\footnotesize $-1$}}
\put(18,42){$n$}
\put(45,42){$0$}
\put(71.5,42){$0$}
\put(99,42){$0$}
\put(26,16.5){$0$}
\put(53.3,16.5){$0$}
\put(81.2,16.5){$0$}
\put(34.5,28.5){\textcolor[cmyk]{1,0.5,0,0}{$-\frac{1}{n}$}}
\put(61,28.5){\textcolor[cmyk]{1,0.5,0,0}{$-\frac{1}{n}$}}
\put(86.5,28.5){\textcolor[cmyk]{1,0.5,0,0}{$-\frac{1}{n}$}}
\put(17,19.5){\textcolor[cmyk]{1,0.5,0,0}{$-1-\frac{1}{n}$}}
\put(45,19.5){\textcolor[cmyk]{1,0.5,0,0}{$-1-\frac{1}{n}$}}
\put(82.5,10){\textcolor[cmyk]{1,0.5,0,0}{$-\frac{1}{n}$}}
\put(58,37.2){\textcolor[cmyk]{0,1,1,0}{$1$}}
\put(60.5,37.2){\textcolor[cmyk]{0,1,1,0}{$-1$}}
\put(86,40){\textcolor[cmyk]{0,1,1,0}{$-1$}}
\put(95.4,30){\textcolor[cmyk]{0,1,1,0}{$1$}}
\put(13,15){\textcolor[cmyk]{0,1,1,0}{$-1$}}
\put(22.5,4){\textcolor[cmyk]{0,1,1,0}{$1$}}
\put(40.3,15){\textcolor[cmyk]{0,1,1,0}{$-1$}}
\put(50.5,4){\textcolor[cmyk]{0,1,1,0}{$1$}}
\put(69,15){\textcolor[cmyk]{0,1,1,0}{$-1$}}
\put(78.5,4){\textcolor[cmyk]{0,1,1,0}{$1$}}
\put(20,33.8){blow}
\put(21,31.5){up}
\put(73.8,33.8){slide}
\put(.5,9){slide}
\put(27,9){{\small isotopy}}
\put(55,8.8){slide}
\end{overpic}
\caption{A proof of  $M_{K}(n) \approx M_{K'}(n)$
when $K = 8_{20}$.}
\label{fig:HM1}
\end{figure}
\begin{figure}[!htb]
\centering
\begin{overpic}[width=.9\textwidth]{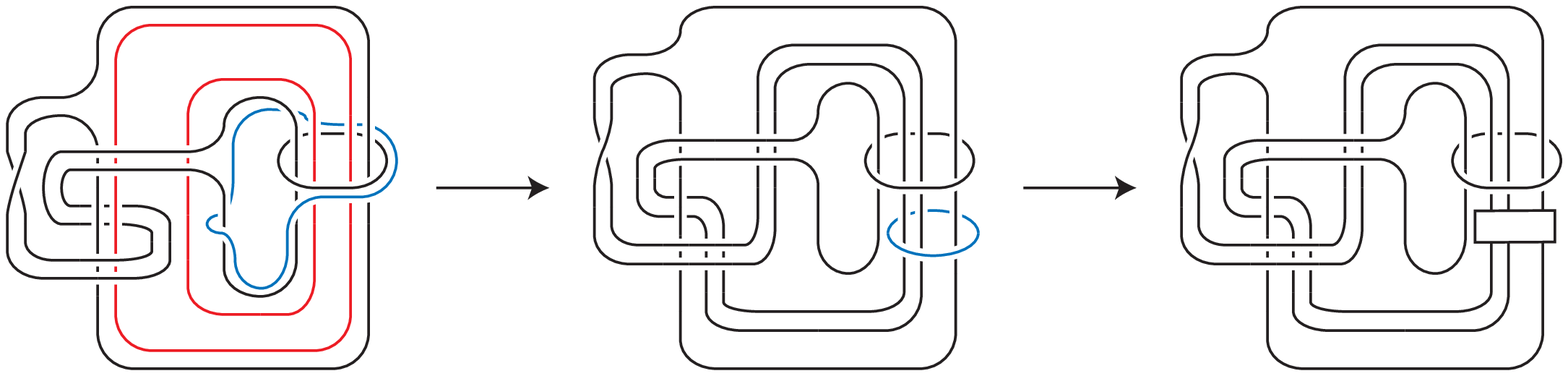}
\put(24,22){$0$}
\put(61,22){$0$}
\put(98.5,22){$n$}
\put(24,16){\textcolor[cmyk]{1,0.5,0,0}{$-\frac{1}{n}$}}
\put(62,7){\textcolor[cmyk]{1,0.5,0,0}{$-\frac{1}{n}$}}
\put(8,20){\textcolor[cmyk]{0,1,1,0}{$-1$}}
\put(20,7.2){\textcolor[cmyk]{0,1,1,0}{$1$}}
\put(15.2,11){{\footnotesize $-1$}}
\put(61.5,14.5){$-1$}
\put(99,14.5){$-1$}
\put(96,8.8){$n$}
\put(65.5,12.5){blow}
\put(65.5,9.7){down}
\end{overpic}
\caption{Moves which correspond to the operation $(*n)$.}
\label{fig:HM2}
\end{figure}

Next we consider a general  case.
Let $(A, b, c)$ be an annulus presentation of $K$. 
As seen in Figure~\ref{fig:HM3}, 
$M_{K}(n) $ is represented by the last diagram in Figure~\ref{fig:HM3}. 
Now it is not difficult to see that this is diffeomorphic to $M_{K'}(n)$. 
\begin{figure}[!htb]
\centering
\begin{overpic}[width=\textwidth]{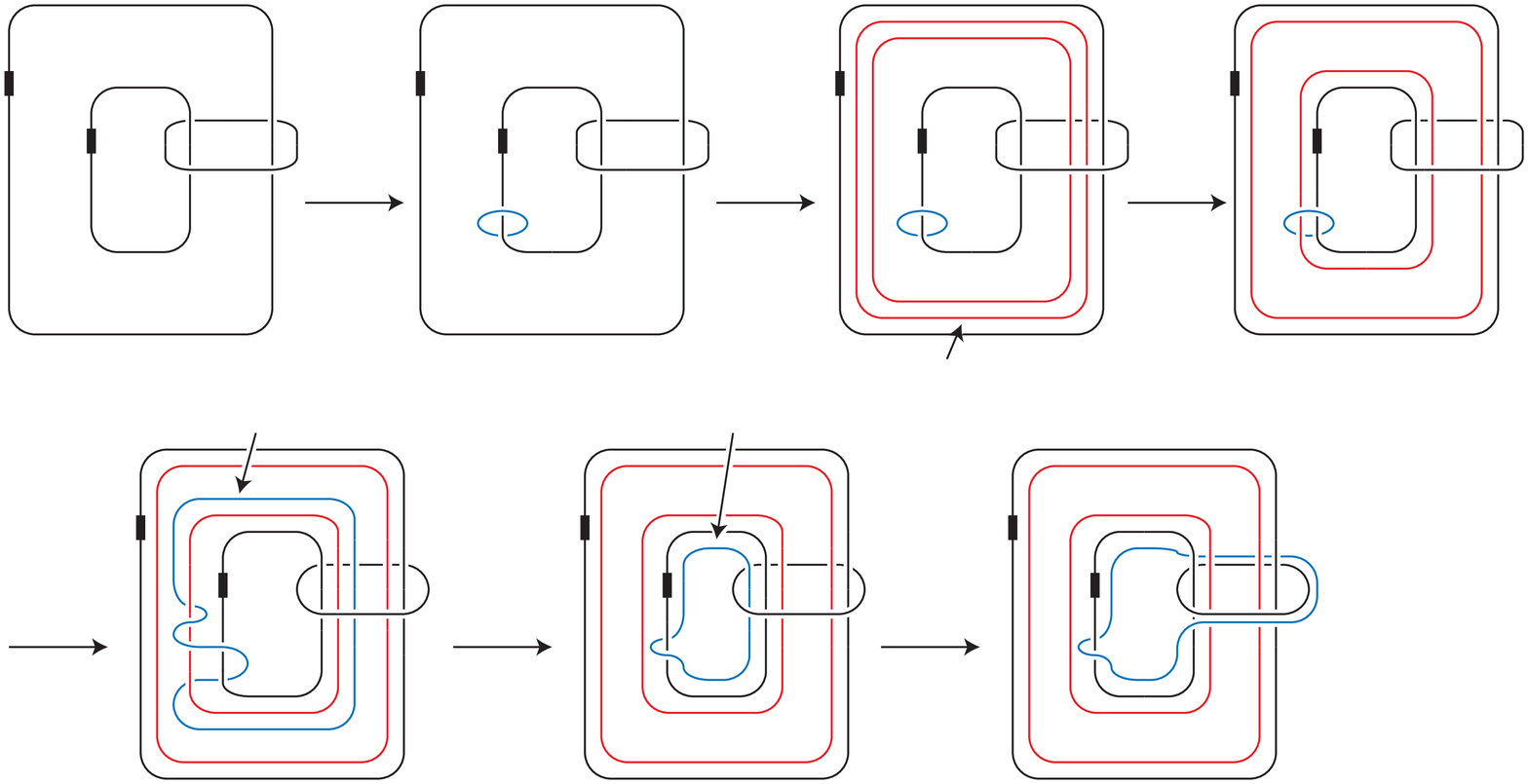}
\put(17.8,50){$n$}
\put(19.5,41.5){$-1$}
\put(44.8,50){$0$}
\put(46.8,41.5){$-1$}
\put(72.5,50){$0$}
\put(74.2,41.5){$-1$}
\put(98.5,50){$0$}
\put(100,41.5){$-1$}
\put(26.5,21){$0$}
\put(28,12){$-1$}
\put(55.5,21){$0$}
\put(57,12){$-1$}
\put(83.5,21){$0$}
\put(73.5,12){$-1$}
\put(34.5,36.5){\textcolor[cmyk]{1,.5,0,0}{$-\frac{1}{n}$}}
\put(62,36.5){\textcolor[cmyk]{1,.5,0,0}{$-\frac{1}{n}$}}
\put(87.3,36.5){\textcolor[cmyk]{1,.5,0,0}{$-\frac{1}{n}$}}
\put(12,24){\textcolor[cmyk]{1,.5,0,0}{$-1-\frac{1}{n}$}}
\put(43,24){\textcolor[cmyk]{1,.5,0,0}{$-1-\frac{1}{n}$}}
\put(86.5,12.5){\textcolor[cmyk]{1,.5,0,0}{$-\frac{1}{n}$}}
\put(61,32){\textcolor[cmyk]{0,1,1,0}{$-1$}}
\put(60.7,27){\textcolor[cmyk]{0,1,1,0}{$1$}}
\put(93,32){\textcolor[cmyk]{0,1,1,0}{$1$}}
\put(82.2,31){\textcolor[cmyk]{0,1,1,0}{$-1$}}
\put(12.6,15){\textcolor[cmyk]{0,1,1,0}{$1$}}
\put(10.5,1.7){\textcolor[cmyk]{0,1,1,0}{$-1$}}
\put(40.5,15){\textcolor[cmyk]{0,1,1,0}{$1$}}
\put(39.8,1.7){\textcolor[cmyk]{0,1,1,0}{$-1$}}
\put(68.6,15){\textcolor[cmyk]{0,1,1,0}{$1$}}
\put(67.8,1.7){\textcolor[cmyk]{0,1,1,0}{$-1$}}
\put(20,35.8){blow}
\put(20,33.5){up}
\put(74,35.8){slide}
\put(1,6.5){slide}
\put(29,6.5){isotopy}
\put(58,6.5){slide}
\end{overpic}
\caption{A proof of  $M_{K}(n) \approx M_{K'}(n)$ for a general case.}
\label{fig:HM3}
\end{figure}
\end{proof}

\begin{rem}\label{rem:simple_annulus_presentation}
Let $K$ be a knot with an annulus presentation $(A, b, c)$
and $K'$ be the knot obtain from $K$ by the operation ($*n$).
In general, $K'$ is much complicated than $K$. 
If the annulus presentation $(A, b, c)$ is simple, then $K'$ is not too complicated. 
Indeed, let $(A, b_A, c)$ be the annulus presentation 
obtained from $(A, b, c)$ by applying the operation $(A)$ 
as in the left side of Figure~\ref{fig:RemOperation}. 
Then the knot $K'$ is indicated as in the right side of Figure~\ref{fig:RemOperation}.
\end{rem}

\begin{figure}[!htb]
\centering
\begin{overpic}[width=.7\textwidth]{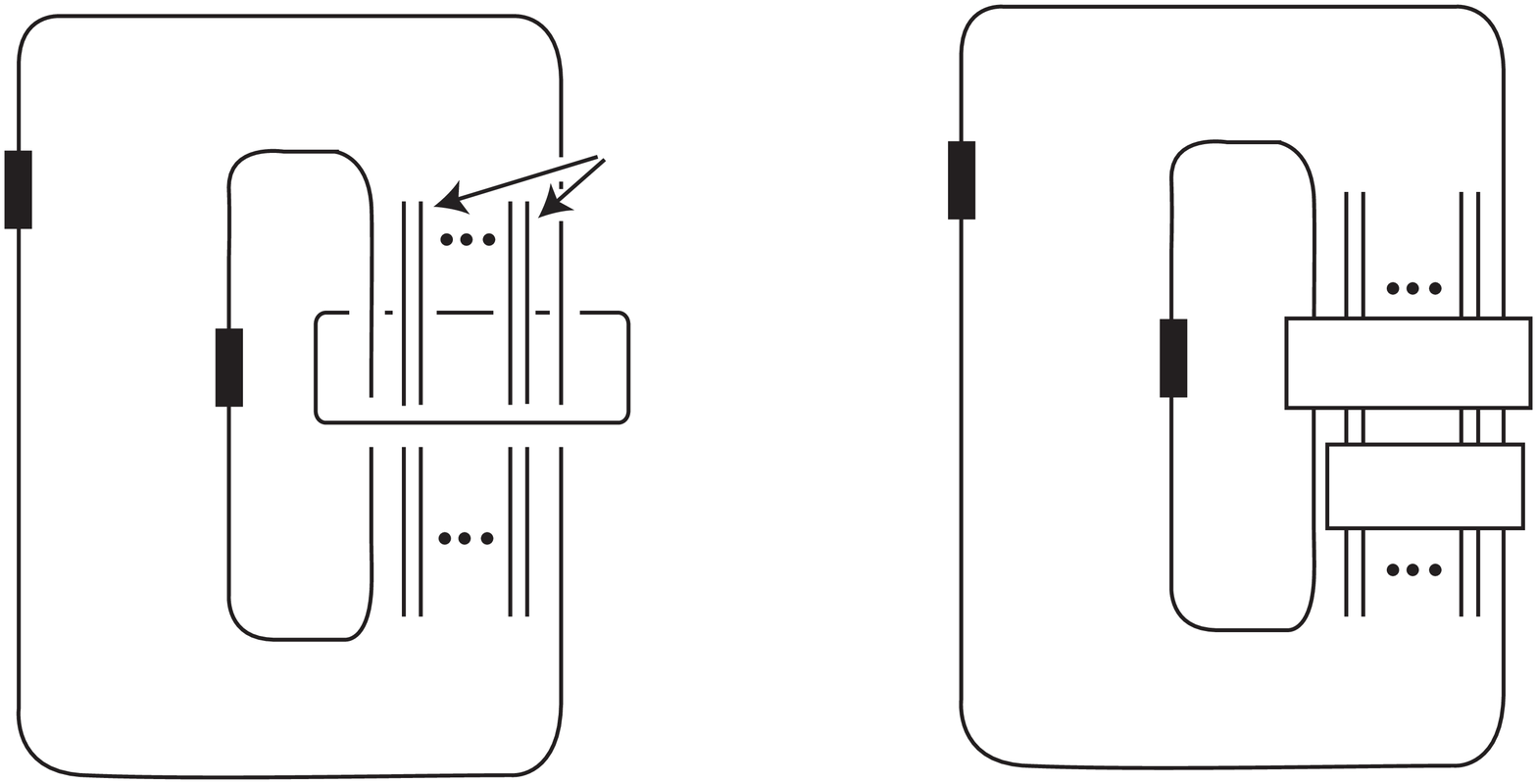}
\put(39.5,41){$b(I \times I)$}
\put(40.5,21){$c$}
\put(41.5,27){$-1$}
\put(89.5,26){$-1$}
\put(92,18.5){$n$}
\end{overpic}
\caption{The annulus presentation $(A, b_A, c)$ and the knot $K'$.}
\label{fig:RemOperation}
\end{figure}

\section{Extension of a diffeomorphism between 3-manifolds}\label{sec:extension} 

In his seminal work,
Cerf~\cite{Cerf} proved that $\Gamma_4=0$, that is, 
any orientation preserving self diffeomorphism of $S^3$
extends to a self diffeomorphism of $B^4$.
As an application of $\Gamma_4=0$, Akbulut obtained the following lemma. 

\begin{lem}[\cite{Ak2}]\label{lem:extend}
Let $K$ and $K'$ be knots in $S^3 = \partial D^4$ 
with a diffeomorphism $g : \partial X_{K}(n) \to \partial X_{K'}(n)$, 
and let $\mu$ be a meridian of $K$. 
Suppose that
\begin{enumerate}
\item 
if $\mu$ is $0$-framed, then $g(\mu)$ is the $0$-framed unknot in the Kirby diagram representing $X_{K'}(n)$, and 
\item 
the Kirby diagram $X_{K'}(n) \cup h^1$  represents $D^4$, where $h^1$ is the $1$-handle represented by  (dotted) $g(\mu)$. 
\end{enumerate}
Then $g$ extends to a diffeomorphism $\widetilde{g}: X_{K}(n) \to X_{K'}(n)$ 
such that $\tilde{g}|_{\partial X_{K}(n)}=g$. 
\end{lem} 

This technique is called ``carving'' in~\cite{AkbulutBook}. 
For a proof, we refer the reader to \cite[Lemma 2.9]{AJOT}. 
Applying Lemma~\ref{lem:extend}, we show the following.

\begin{theorem}\label{thm:diffeo4} 
Let $K$ be a knot with a simple annulus presentation 
and $K'$ be the knot obtain from $K$ by the operation $(*n)$. 
Then $X_{K}(n) \approx X_{K'}(n)$.  
\end{theorem}
\begin{proof}
First, 
we consider the case where $K = 8_{20}$ with the usual simple annulus presentation.
Let $f : \partial X_{K}(n) \to  \partial X_{K'}(n)$ 
be the diffeomorphism given in Figures~\ref{fig:HM1} 
and \ref{fig:HM2}. 
Let $\mu$ be the meridian of $K$. 
If we suppose that $\mu$ is $0$-framed, 
then we can check that $f(\mu)$ is the $0$-framed unknot 
in the Kirby diagram of $X_{K'}(0)$ as in Figure~\ref{fig:Proof4-1}. 
Let $W$ be the $4$-manifold $D^4 \cup h^1 \cup h^2$, 
where $h^1$ is the dotted $1$-handle represented by $f(\mu)$ and $h^2$ 
is the $2$-handle represented by $K'$ with framing $n$.
Sliding $h^2$ over $h^1$, we obtain a canceling pair 
(see Figure~\ref{fig:Proof4-2}), implying that $W \approx B^4$.
By Lemma~\ref{lem:extend}, we have $\tilde{f} : X_{K}(0) \approx X_{K'}(0)$. 

Next, We consider a general case.
Let $g : \partial X_{K}(n) \to  \partial X_{K'}(n)$ 
be the diffeomorphism given in the proof of Theorem~\ref{thm:diffeo3} in a general case  (see Figure~\ref{fig:Proof4-3}), 
and $\mu$ the meridian of $\partial X_{K}(n)$. 
In Figure~\ref{fig:Proof4-3}, 
the annulus presentation in the right hand side represents $K'$, 
see Remark~\ref{rem:simple_annulus_presentation}. 
If we suppose that $\mu$ is $0$-framed, 
then we can check that $g(\mu)$ is the $0$-framed unknot 
in the Kirby diagram of $X_{K'}(0)$ as in Figure~\ref{fig:Proof4-3}. 
Let $W$ be the $4$-manifold $D^4 \cup h^1 \cup h^2$, 
where $h^1$ is the dotted $1$-handle represented by $g(\mu)$ and $h^2$ 
is the $2$-handle represented by $K'$ with framing $n$. 
Sliding $h^2$ over $h^1$, we obtain a canceling pair 
(see Figure~\ref{fig:Proof4-4}), implying that $W \approx B^4$.
By Lemma~\ref{lem:extend} again, we have $\tilde{g} : X_{K}(0) \approx X_{K'}(0)$.
\end{proof}

\begin{figure}[!htb]
\centering
\begin{overpic}[width=.75\textwidth]{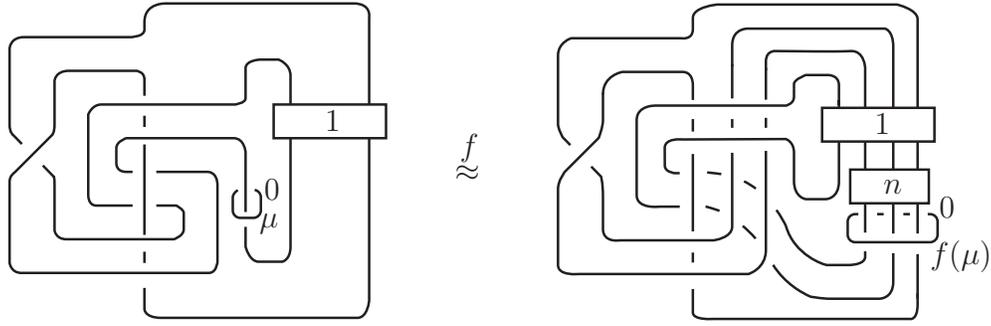}
\put(27.5,13){$0$}
\put(27,10){$\mu$}
\put(34,20.3){$1$}
\put(93,20){$1$}
\put(94,13.6){$n$}
\put(100,11){$0$}
\put(99,6){$f(\mu)$}
\put(48,15){$\approx$}
\put(48.5,17.7){$f$}
\end{overpic}
\caption{The image of $\mu$ under $f$.}
\label{fig:Proof4-1}
\end{figure}

\begin{figure}[!htb]
\centering
\begin{overpic}[width=.75\textwidth]{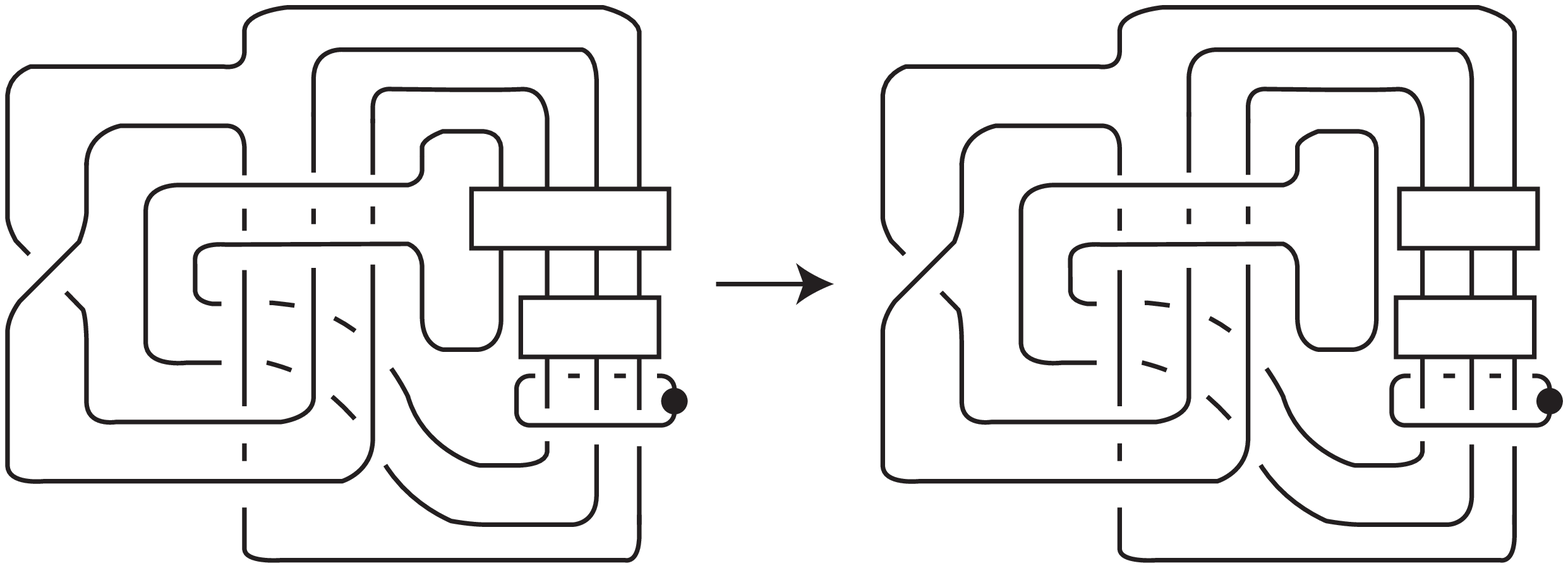}
\put(35.8,21){$1$}
\put(37,14.5){$n$}
\put(40.5,35){$n$}
\put(97,35){$n-2$}
\put(93.3,21){$1$}
\put(93,14.5){$n$}
\put(103,17){$\approx$}
\put(107,17){$B^4$}
\end{overpic}
\caption{The $4$-manifold $W$ is diffeomorphic to $B^4$.}
\label{fig:Proof4-2}
\end{figure}

\begin{figure}[!htb]
\centering
\begin{overpic}[width=.5\textwidth]{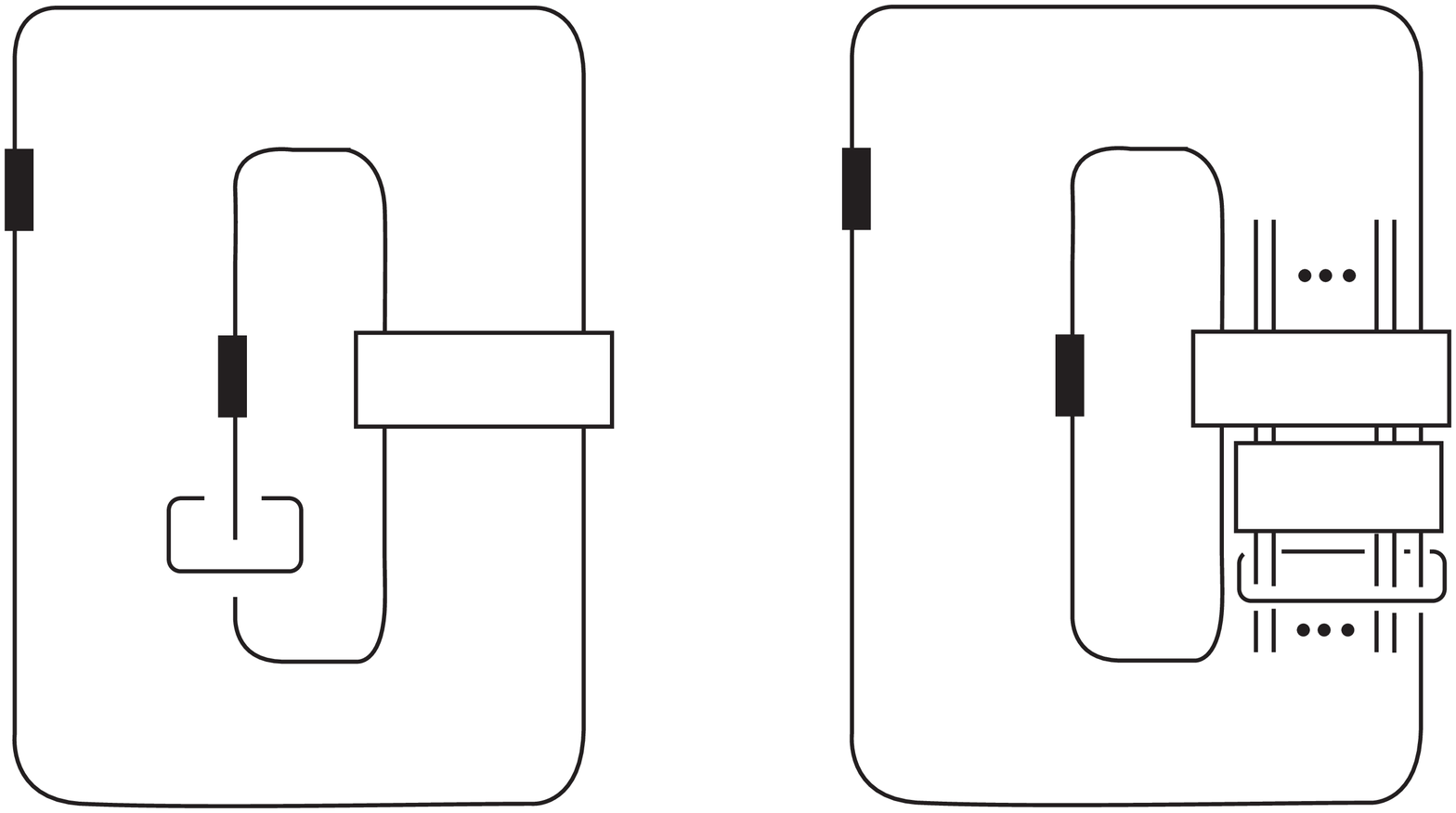}
\put(21,20){$0$}
\put(20.7,15){$\mu$}
\put(32.5,28){$1$}
\put(40,53){$n$}
\put(100.2,16.7){$0$}
\put(100,11.7){$g(\mu)$}
\put(90,28){$1$}
\put(91,21){$n$}
\put(98,53){$n$}
\put(48,27){$\approx$}
\end{overpic}
\caption{The image of $\mu$ under $g$.}
\label{fig:Proof4-3}
\end{figure}

\begin{figure}[!htb]
\centering
\begin{overpic}[width=.5\textwidth]{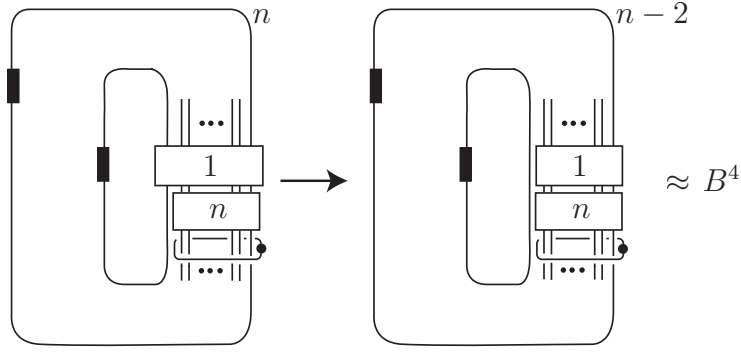}
\put(31.5,27.3){$1$}
\put(32.5,20.6){$n$}
\put(39.5,52){$n$}
\put(91,27.3){$1$}
\put(91,20.6){$n$}
\put(98,52){$n-2$}
\put(106,25){$\approx$}
\put(112,25){$B^4$}
\end{overpic}
\caption{The 4-manifold $W$ is diffeomorphic to $B^4$.}
\label{fig:Proof4-4}
\end{figure}

\begin{rem}
It is important which knot admits a simple annulus presentation. 
An answer is a knot with unknotting number one (see \cite[Lemma 2.2]{AJOT}). 
\end{rem}

\section{Proof of Theorem~\ref{thm:main}}\label{sec:proof}

For a knot $K$, 
we denote  by $\Delta_{K}(t)$  the Alexander polynomial of $K$.
We assume that $\Delta_{K}(t)$ is of the symmetric form 
\[ \Delta_{K}(t) = a_0 + \sum_{i=1}^d a_i(t^i + t^{-i}) \, . \]
We call the integer $d$ the \emph{degree} of $\Delta_{K}(t)$, 
and denote it by $\deg \Delta_{K}(t)$. 

In this section, we define a ``good'' annulus presentation. 
Theorem~\ref{thm:main} will be shown as a typical case of the argument in this section. 
The following technical lemma plays an important role. 

\begin{lem}\label{lem:technical}
Let $n$ be a positive integer. 
Let $K$ be a knot with a good annulus presentation, 
and $K'$ be the knot obtained from $K$ by applying the operation $(*n)$. 
Then 
\begin{enumerate}[{\rm (i)}]
\item 
$K'$ also admits a good annulus presentation, and 
\item 
$\deg \Delta_K(t) < \deg \Delta_{K'}(t)$.
\end{enumerate}
\end{lem}

We will prove Lemma~\ref{lem:technical} later. 
Using Lemma~\ref{lem:technical}, we show the following 
which yields Theorem~\ref{thm:main} as an immediate corollary. 

\begin{theorem}\label{thm:main2}
Let $n$ be a positive integer. 
Let $K_{0}$ be a knot with a good annulus presentation 
and $K_{i}$ $( i \ge 1)$ the knot obtained from $K_{i-1}$ by applying the operation $(*n)$. 
Then 
\begin{enumerate}
\item 
$X_{K_0}(n) \approx X_{K_1}(n) \approx X_{K_2}(n)   \approx \cdots$, and 
\item 
the knots $K_{0}, K_{1}, K_{2}, \cdots$ are mutually distinct.
\end{enumerate}
Let $\overline{K}_{i}$ be the mirror image of $K_{i}$. 
Then 
\begin{enumerate}
\setcounter{enumi}{2}
\item 
$X_{\overline{K}_{0}}(-n) \approx X_{\overline{K}_{1}}(-n) \approx X_{\overline{K}_{2}}(-n) \approx \cdots$, and 
\item 
the knots $\overline{K}_{0}, \overline{K}_{1}, \overline{K}_{2}, \cdots$ are mutually distinct.
\end{enumerate}
\end{theorem}
\begin{proof}
By the definition (Definition~\ref{def:good}), 
any good annulus presentation is simple. 
Thus, by Theorem~\ref{thm:diffeo4}, we have 
\[ X_{K_0}(n) \approx X_{K_1}(n) \approx X_{K_2}(n)   \approx \cdots \, .\]
By Lemma ~\ref{lem:technical} (i),
each $K_{i}$ $( i \ge 1)$ also admits a good annulus presentation. 
Thus, by Lemma~\ref{lem:technical} (ii), we have 
\[ \deg \Delta_{K_0}(t)  < \deg \Delta_{K_1}(t)  < \deg \Delta_{K_2}(t) < \cdots \, . \]
This implies that the knots $K_{0}, K_{1}, K_{2}, \cdots$ are mutually distinct. 

Since $X_{K_i}(n) \approx X_{\overline{K}_i}(-n)$ 
and $\deg \Delta_{K_i}(t)=\deg \Delta_{\overline{K}_i}(t)$, 
we have
\[ X_{\overline{K}_0}(-n) \approx X_{\overline{K}_1}(-n) \approx X_{\overline{K}_2}(-n)   \approx \cdots \, , \  \text{and}\]
\[ \deg \Delta_{\overline{K}_0}(t)  < \deg \Delta_{\overline{K}_1}(t)  < \deg \Delta_{\overline{K}_2}(t) < \cdots \, .\]
This completes the proof of Theorem~\ref{thm:main2}. 
\end{proof}

\subsection{Good annulus presentation and the Alexander polynomial}

Let $K$ be a knot with a simple annulus presentation $(A, b, c)$. 
Note that the knot $(\partial A \setminus b(\partial I \times I) )\cup b(I \times \partial I)$ is 
trivial\footnote{$K$ is the knot $(\partial A \setminus b(\partial \times I) )\cup b(I \times \partial I)$ in $M_c(-1)$.} in $S^3$ 
if we ignore the $(-1)$-framed loop $c$. 
We denote by $U$ this trivial knot. 
Since $(A,b, c)$ is simple, $U \cup c$ can be isotoped so that 
$U$ bounds a ``flat'' disk $D$ (contained in $\R^2 \cup \{\infty\}$). 
This isotopy, denote by $\varphi_b$, is realized by shrinking the band $b(I \times I)$. 
For example, see Figure~\ref{fig:ExStandard}. 
In the abbreviated form, 
$\varphi_b$ is represented as in Figure~\ref{fig:StandardAbb}. 
Here we note that the linking number of $U$ and $c$ is zero 
since we assumed that $A \cup b(I \times I)$ is orientable (see subsection~\ref{subsec:AP}). 
Let $\Sigma$ be the disk bounded by $c$ as in Figure~\ref{fig:StandardAbb}. 
We assume that $\Sigma$ stays during the isotopy $\varphi_b$. 

\begin{figure}[!htb]
\centering
\begin{overpic}[width=.75\textwidth]{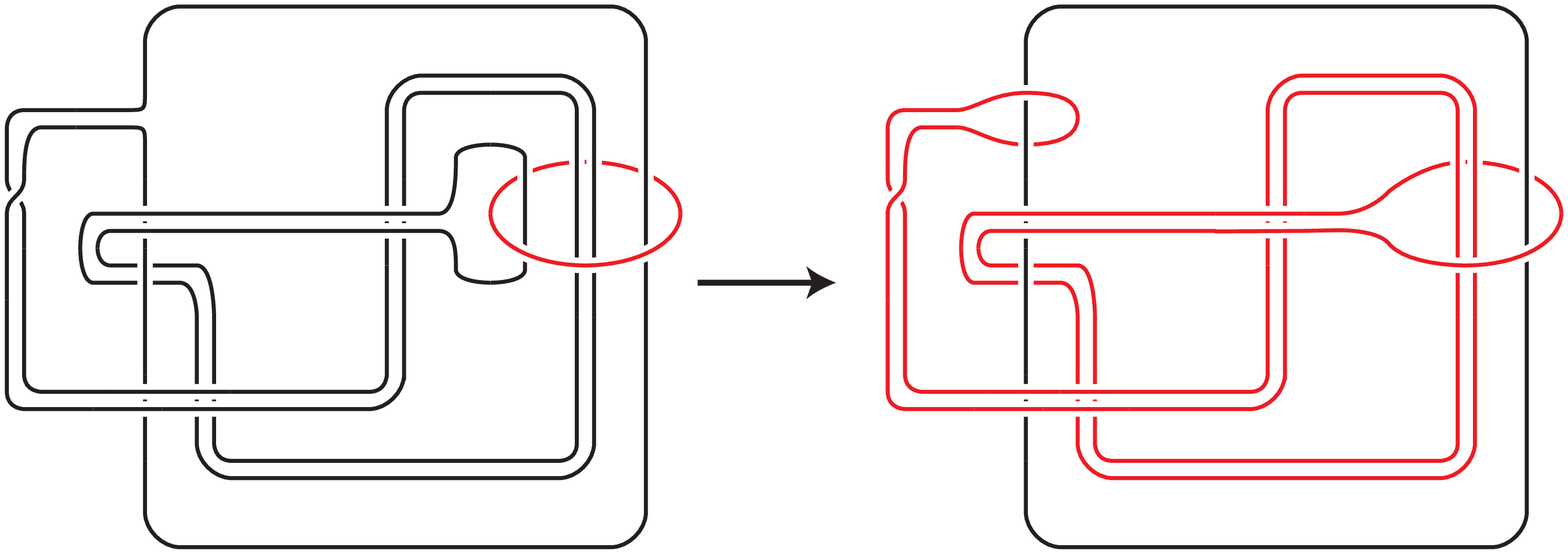}
\put(43,23.5){$\textcolor[cmyk]{0,1,1,0}{c}$}
\put(41,34){$U$}
\put(82.5,13){$\textcolor[cmyk]{0,1,1,0}{c}$}
\put(97.5,34){$U$}
\put(43.8,19){isotopy}
\put(47.5,14.5){$\varphi_b$}
\end{overpic}
\caption{By the isotopy $\varphi_b$ (shrinking the band $b(I \times I)$), 
$U \cup c$ (the left side) is changed to the right side.}
\label{fig:ExStandard}
\end{figure}
\begin{figure}[!htb]
\centering
\begin{overpic}[width=.75\textwidth]{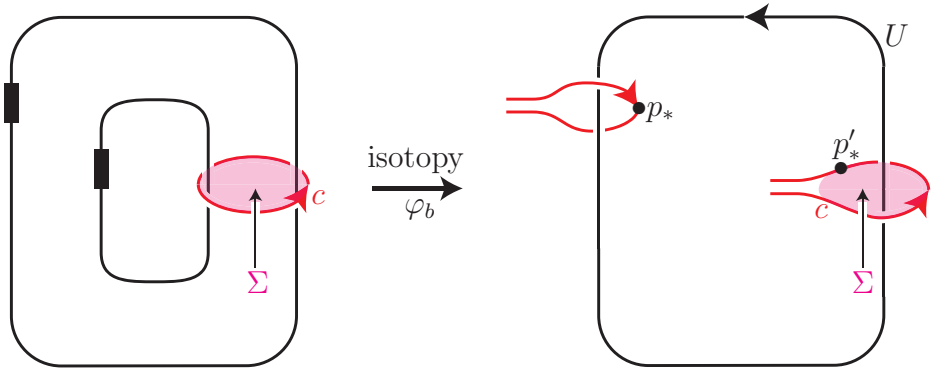}
\put(33,18){$\textcolor[cmyk]{0,1,1,0}{c}$}
\put(87,16.5){$\textcolor[cmyk]{0,1,1,0}{c}$}
\put(94.5,35){$U$}
\put(39,22){isotopy}
\put(43,17){$\varphi_b$}
\put(67.3,27.5){$\bullet$}
\put(89,21){$\bullet$}
\put(69,27.5){$p_*$}
\put(89,23.5){$p_*'$}
\put(26,8.5){\textcolor[cmyk]{0,1,0,0}{$\Sigma$}}
\put(91,8.5){\textcolor[cmyk]{0,1,0,0}{$\Sigma$}}
\end{overpic}
\caption{The isotopy $\varphi_b$ in the abbreviated form of $(A,b,c)$. 
Assume that $\Sigma$ stays during the isotopy.}
\label{fig:StandardAbb}
\end{figure}

After the isotopy $\varphi_b$, cutting along the disk $D$, 
the loop $c$ is separated into arcs whose endpoints are in $D$. 
Furthermore, choosing orientations on $c$ and $U$, these arcs are oriented. 
Unless otherwise noted, 
we choose the orientations of $c$ and $U$ as in Figure~\ref{fig:StandardAbb}. 
These oriented arcs are classified into four types as follows: 
For $p \in c \cap D$, let $\sign(p) = \pm$ 
according to the sign of the intersection between $D$ and $c$ at $p$. 
For an oriented arc $\alpha$, let $p_s$ (resp.~$p_t$) 
be the starting point (resp.~terminal point) of $\alpha$. 
Then we say that $\alpha$ is of \emph{type} $(\sign(p_s) \sign(p_t))$. 
That is, the oriented arc $\alpha$ is of type  $(++)$, $(--)$, $(+-)$, or $(-+)$. 
For example, see Figure~\ref{fig:ExStandard2}.

\begin{figure}[!htb]
\centering
\begin{overpic}[width=.4\textwidth]{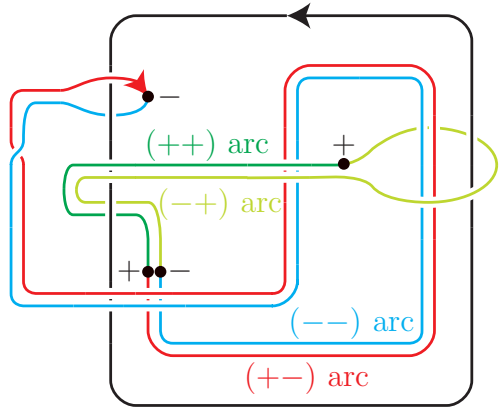}
\put(30.6,61.7){$-$}
\put(65.7,52){$+$}
\put(32.7,26.5){$-$}
\put(22.5,26.5){$+$}
\put(48,5){\textcolor[cmyk]{0,1,1,0}{$(+-)$ arc}}
\put(57,16){\textcolor[cmyk]{1,0,0,0}{$(--)$ arc}}
\put(30.5,40){\textcolor[cmyk]{.3,0,1,0}{$(-+)$ arc}}
\put(28,52){\textcolor[cmyk]{1,0,1,0}{$(++)$ arc}}
\end{overpic}
\caption{The four types of arcs. }
\label{fig:ExStandard2}
\end{figure}

Here we consider the infinite cyclic covering $\tilde E(U)$ of $E(U)$.
Notice that $\tilde E(U)$ consists of infinitely many copies of a cylinder 
obtained from $E(U)$ by cutting along $D$. 
Thus $\tilde E(U)$ is diffeomorphic to $D \times \R \approx \cup_{i \in \Z} \left( D \times [i,i+1] \right)$. 
Each oriented arc is lifted in $\tilde E(U)$ as shown in Figure~\ref{fig:LiftOfArc}. 
Hereafter, for simplicity, we say an arc instead of an oriented arc.

\begin{figure}[!htb]
\bigskip
\begin{overpic}[width=.7\textwidth]{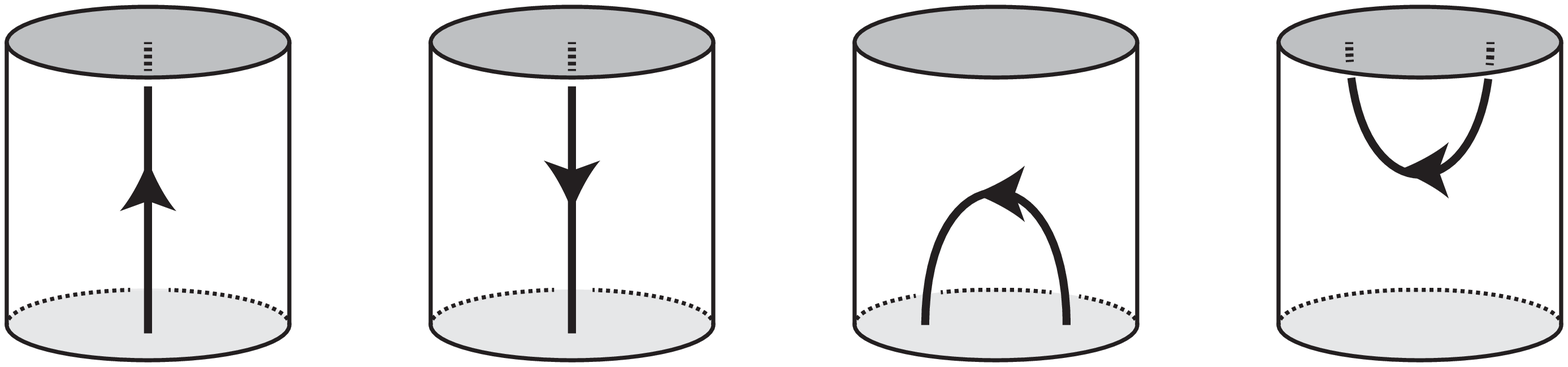}
\put(3,-3.5){$D^2 \times \{ i\}$}
\put(0,24){$D^2 \times \{ i + 1\}$}
\end{overpic}
\medskip
\caption{Lifts of oriented arcs of type $(++), (--), (+-)$, and $(-+)$ respectively. }
\label{fig:LiftOfArc}
\end{figure}
\begin{figure}[!htb]
\begin{overpic}[width=.15\textwidth]{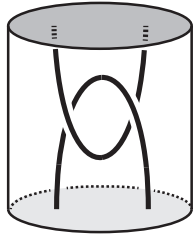}
\end{overpic}
\caption{Lifts of the arcs of type $(+-)$ and $(-+)$ from a good annulus presentation. }
\label{fig:LiftOfArc2}
\end{figure}

\begin{definition}\label{def:good}
We say that a simple annulus presentation $(A, b, c)$ is \emph{good} 
if the set of arcs $\mathcal A$ obtained as above satisfies the following up to isotopy. 
\begin{enumerate}
\item 
$\mathcal A$ contains just one $(+-)$ arc and one $(-+)$ arc, 
and they are lifted as in Figure~\ref{fig:LiftOfArc2}. 
\item 
For $\alpha \in \mathcal A$, 
if $\Sigma \cap \alpha \ne \emptyset$, 
then $\alpha$ is of type $(++)$ (resp.~$(--)$ arc) 
and the sign of each intersection point in $\Sigma \cap \alpha$ is $+$ (resp.~$-$). 
\item 
$b(I \times \partial I) \cap \mathrm{int}A \ne \emptyset$. 
\end{enumerate}
\end{definition}

\begin{rem}\label{rem:good}
For a simple annulus presentation $(A,b,c)$, after the isotopy $\varphi_b$, 
the intersection $c \cap D$ corresponds to 
the intersection $b(I \times \partial I) \cap \mathrm{int} A$ 
and further two points $p_*$ and $p_*'$ depicted in Figure~\ref{fig:StandardAbb}. 
Notice that 
\[ b(I \times \partial I) \cap \mathrm{int} A = \sqcup_i \, b(\{t_i\} \times \partial I) \]  
for some $0 < t_1 < \dots < t_r < 1$. 
For each $i$, $b(\{t_i\} \times \partial I)$ consists of two points whose signs are differ. 
Furthermore, with the orientation as in Figure~\ref{fig:StandardAbb}, we have 
\[ \sign(p_*) = - \, \text{ \, and \, } \sign(p_*') = + \, . \] 
\end{rem}

\begin{ex}\label{ex:GoodIsotopy}
The annulus presentation obtained by applying $\varphi_b$ 
on Figure~\ref{fig:ExStandard2} is not good since the condition (2) does not hold. 
In such a case, changing the position of an intersection as in Figure~\ref{fig:ExStandard3} by an isotopy, 
we can obtain a good annulus presentation. 
We often apply such an argument in the proof of Lemma~\ref{lem:technical}. 
\begin{figure}[!htb]
\centering
\begin{overpic}[width=.9\textwidth]{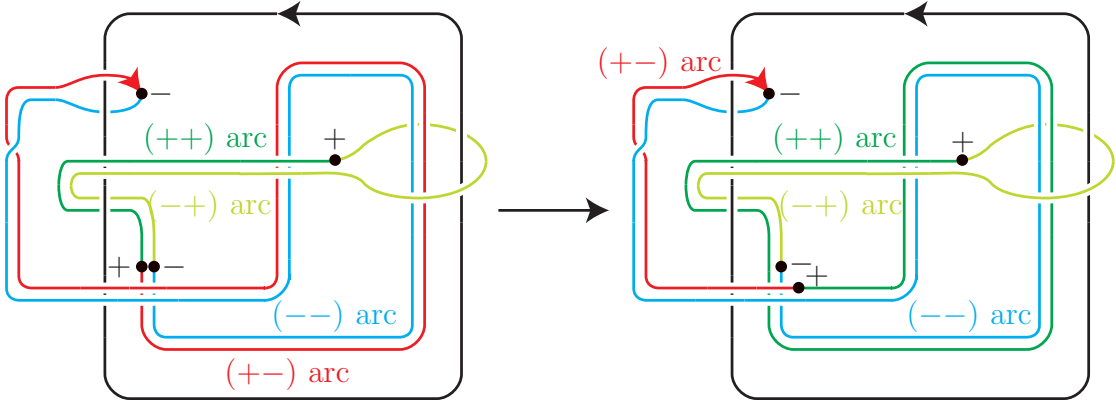}
\put(13.2,27){$-$}
\put(28.8,23){$+$}
\put(14.4,11.6){$-$}
\put(9.5,11.6){$+$}
\put(20,2){\textcolor[cmyk]{0,1,1,0}{$(+-)$ arc}}
\put(24.2,7.2){\textcolor[cmyk]{1,0,0,0}{$(--)$ arc}}
\put(13,17){\textcolor[cmyk]{.3,0,1,0}{$(-+)$ arc}}
\put(12.7,23){\textcolor[cmyk]{1,0,1,0}{$(++)$ arc}}
\put(69.5,27){$-$}
\put(85,22.7){$+$}
\put(70.5,11.8){$-$}
\put(71.9,10.8){$+$}
\put(53.3,29.8){\textcolor[cmyk]{0,1,1,0}{$(+-)$ arc}}
\put(81,7.2){\textcolor[cmyk]{1,0,0,0}{$(--)$ arc}}
\put(69.5,17){\textcolor[cmyk]{.3,0,1,0}{$(-+)$ arc}}
\put(69,23){\textcolor[cmyk]{1,0,1,0}{$(++)$ arc}}
\end{overpic}
\caption{By an isotopy, we move the intersection point of $c \cap D$. }
\label{fig:ExStandard3}
\end{figure}
\end{ex}

Considering a surgery description of the infinite cyclic covering of 
the exterior of $K$, we can easily show the following. 

\begin{lem}\label{lem:LiftOfArc} 
If a knot $K$ admits a good annulus presentation, then 
\begin{align}\label{eq:ArcDeg}
\deg \Delta_K(t) = \# \{ \text{arcs of type } (++) \} + 1 \, .
\end{align}
\end{lem}

For the details of a surgery description of $\tilde{E}(K)$ and the Alexander polynomial, 
we refer the reader to Rolfsen's book~\cite[Chapter 7]{Rolfsen}. 

\begin{rem}
To show Lemma~\ref{lem:LiftOfArc}, 
we do not need the conditions (2) and (3) in Defitnition~\ref{def:good}. 
These conditions are used to prove Lemma~\ref{lem:technical}. 
\end{rem}

\begin{rem}\label{rem:monic}
If a knot $K$ admits a good annulus presentation, 
then we can see that $\Delta_K(t)$ is monic. 
\end{rem}

Now we are ready to prove the main result in this paper. 

\begin{proof}[Proof of Theorem~\ref{thm:main}]
The case where $n=0$ was proved in \cite{AJOT}. 
We can check that the simple annulus presentation of the knot $8_{20}$ 
in Figure~\ref{fig:Def-AP} is good. 
Thus the proof for the case where $n \neq 0$ is obtained by Theorem~\ref{thm:main2} immediately. 
\end{proof}

\subsection{Proof of Lemma~\ref{lem:technical}}

We start the proof of Lemma~\ref{lem:technical}. 
Let $(A, b, c)$ be a good annulus presentation of a knot $K$. 
Recall that the operation $(*n)$ is a composition of the two operations $(A)$ and $(T_n)$ 
for an annulus presentation. 
Let $(A, b_A, c)$ be the annulus presentation obtained from $(A, b, c)$ 
by applying the operation $(A)$, 
and $(A, b', c)$ the annulus presentation obtained from $(A, b_A, c)$ 
by applying the operation $(T_n)$. 
That is,
\[ (A, b, c) \overset{(A)}{\longrightarrow }  (A, b_A, c)  \overset{(T_n)}{\longrightarrow } (A, b', c). \]
Note that $K'$ admits the annulus presentation $(A, b', c)$.

First we show that $(A, b_A,c)$ is good. 
The operation $(A)$ preserves the number of arcs and type of each arc. 
Furthermore we can suppose that the $(+-)$ arc and $(-+)$ arc are fixed by the operation $(A)$ up to isotopy. 
Therefore $(A, b_A, c)$ satisfies the condition (1) of Definition~\ref{def:good}. 
We can also check that $(A, b_A, c)$ satisfies 
the conditions (2) and (3) of Definition~\ref{def:good}. 
Therefore $(A, b_A, c)$ is good. 

Next we show that $(A, b', c)$ is good. 
The operation $(T_n)$ may increase the number of arcs. 
Indeed a $(++)$ (resp.~$(--)$) arc through $\Sigma$ is changed to $n+1$ $(++)$ (resp.~$(--)$) arcs since $(A, b_A, c)$ is good, 
in particular, a $(++)$ arc (resp.~$(--)$ arc) intersects $\Sigma$ 
positively (resp.~negatively). 
Note that the $(+-)$ arc and the $(-+)$ arc are fixed by the operation $(T_n)$. 
Hence $(+-)$ arcs and $(-+)$ arcs do not produced by the operation $(T_n)$. 
Therefore $(A, b', c)$ satisfies the condition (1). 
We can also check that $(A, b',c)$ satisfies the conditions (2) and (3). 
Therefore $(A, b',c)$ (of $K'$) is good. 
This completes the proof of the claim (i) of Lemma~\ref{lem:technical}. 

Let $ \delta = \# \left(A \cap b(I \times \partial I) \right) / 2$ and 
$\sigma = \# \left( \Sigma \cap b(I \times \partial I) \right) / 2$. 
Then we see that 
\[ \# (A \cap b_A (I \times \partial I)) /2 = \delta \, , \qquad 
\# (\Sigma \cap b_A(I \times \partial I)) /2 = \sigma + \delta \, . \]
Then we have 
\begin{align*}
\# (A \cap b'(I \times \partial I)) /2 
& = \# (A \cap b_A(I \times \partial I))/2 + n \cdot \# (\Sigma \cap b_A(I \times \partial I)) /2 \\ 
&= (n+1) \delta + n \sigma \, , 
\end{align*}
and 
\begin{align*}
\# (\Sigma \cap b'(I \times \partial I)) 
&= \# (\Sigma \cap b_A(I \times \partial I)) \, .
\end{align*}
These are equivalent to 
\begin{align*}
\begin{pmatrix}
\delta' \\ 
\sigma'
\end{pmatrix}
= 
\begin{pmatrix}
n+1 & n \\ 
1 & 1 
\end{pmatrix}
\begin{pmatrix}
\delta \\ 
\sigma
\end{pmatrix}.  
\end{align*}
Since $n \ge 1$ and $\delta \ge 1$, we have 
\begin{align}\label{eq:AlphaIneq}
\delta < \delta' \, . 
\end{align}
By the condition that $(A, b, c)$ and $(A, b',c)$ is good, 
and by Remark~\ref{rem:good}, we see that 
\[ \delta = \# \set{(++)\text{ arcs of }(A, b, c)} \, , \qquad 
\delta' = \# \set{(++)\text{ arcs of }(A, b', c)} \, . \] 
Therefore, by Lemma~\ref{lem:LiftOfArc}, we have 
\begin{align}\label{eq:DegSing}
\deg \Delta_K = \delta + 1 \, ,  \quad \deg \Delta_{K'} = \delta' + 1 \, . 
\end{align}
By \eqref{eq:AlphaIneq} and \eqref{eq:DegSing}, we have 
$\deg \Delta_{K}(t) < \deg \Delta_{K'}(t)$. 
This completes the proof of the claim (ii) of Lemma~\ref{lem:technical}, 
and thus, the proof of Lemma~\ref{lem:technical}.

\appendix
\section{Potential application}\label{sec:Cabling-conjecture} 

We introduce a potential application of our technique to the cabling conjecture. 

\begin{conj}[cabling conjecture, \cite{GS}]\label{conj:cabling}
Let $K$ be a knot in $S^3$. 
If $M_{K}(\gamma)$ is a reducible manifold for some integer $\gamma$, 
then $K$ is a $(p, q)$-cable of a knot and $\gamma=pq$. 
\end{conj}

It is known that this conjecture is true for many knots, 
in particular, torus knots and satellite knots. 
Therefore if any hyperbolic knot admits no reducible surgery, 
then Conjecture~\ref{conj:cabling} is true. 
For details, we refer the reader to \cite{MO}, \cite{Bo}. 
Here we explain a potential application of Theorem~\ref{thm:diffeo3} 
to Conjecture~\ref{conj:cabling}. 
Suppose that $K$ is a $(p, q)$-cable of a knot admitting an annulus presentation, 
and $K'$ the knot obtain from $K$ by the operation ($*{pq}$). 
Then $M_{K'}(pq)$ is a reducible manifold 
since $M_{K}(pq)$ is reducible and $M_{K}(pq) \approx M_{K'}(pq)$ 
by Theorem~\ref{thm:diffeo3}. 
If $K'$ is a hyperbolic knot, 
then it is a counterexample to the cabling conjecture. 
Therefore it is interesting to determine whether 
a $(p, q)$-cable of a knot admits an annulus presentation or not.

Recall that a $(p,q)$-torus knot is a $(p,q)$-cable of the trivial knot. 

\begin{lem}\label{lem:trefoil}
A torus knot admits an annulus presentation 
with $\varepsilon = -1$ (resp.~$\varepsilon = 1$) 
if and only if it is the unknot or the negative (reps.~positive) trefoil knot .
\end{lem}
\begin{proof} 
We only show the case where $\varepsilon = -1$ 
since the proof for the case where $\varepsilon = 1$ is achieved in a similar way. 
Let $T$ be a torus knot which admits 
an annulus presentation and $H$ the negative Hopf link.
Then  $T$ and $H$
are related by a single band surgery.
Therefore 
\[ |\sigma(T) - \sigma(H)| \le 1.\]
That is, 
\[ 0 \le \sigma(T)  \le 2.\]
This implies that $T$ is the unknot or the negative trefoil knot.
On the other hand, 
the unknot and the negative trefoil knot have annulus presentations,
see Figure~\ref{fig:UnknotTrefoil}.
\end{proof}

\begin{figure}[!htb]
\centering
\begin{overpic}[width=.7\textwidth]{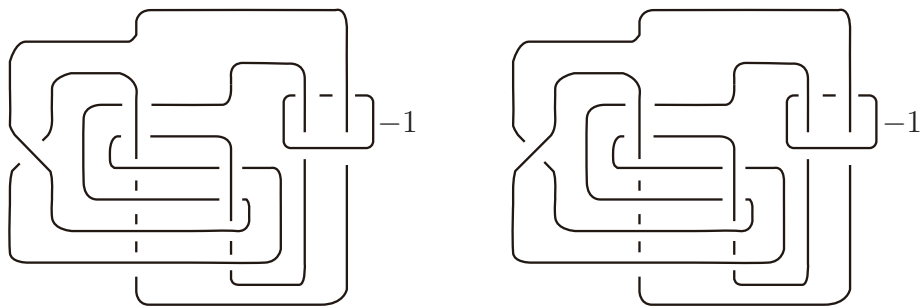}
\put(42.5,20){$-1$}
\put(100.5,20){$-1$}
\end{overpic}
\caption{Annulus presentations of the unknot and the negative trefoil knot.}
\label{fig:UnknotTrefoil}
\end{figure}


Let $K$ be the unknot (resp.~the negative trefoil knot). 
If $M_{K}(\gamma) \approx M_{K'}(\gamma)$ 
for some knot $K'$ and an integer $\gamma$, 
then $K'$ is the unknot (resp.~the trefoil knot), see~\cite{OS}. 
Therefore, by using an annulus presentation of the trivial knot, 
we can not obtain a counterexample of the cabling conjecture 
by using Theorem~\ref{thm:diffeo3} unfortunately. 
Then we propose the following question.

\begin{ques}
Let $K$ be a $(p,q)$-cable knot of a non-trivial knot.
Then does $K$ admit any annulus presentation? 
\end{ques}

\begin{rem}
If the $4$-ball genus of a knot $K$ is greater than 
one, 
then $K$ does not admit any annulus presentations.
Therefore, for example, 
the (2,1)-cable of the trefoil knot does not admit
any annulus presentations.
On the other hand, 
it is not known whether the $(2,1)$-cable of the figure-eight  admits
an annulus presentation or not.
\end{rem}


\begin{thebibliography}{100}


\bibitem{AJOT} 
T. Abe, I. D. Jong, Y. Omae, and M. Takeuchi,
\emph{Annulus twist and diffeomorphic 4-manifolds}, 
Math. Proc. Cambridge Philos. Soc.
\textbf{155} (2013), 219--235.

\bibitem{AT} 
T. Abe and M. Tange,
\emph{A construction of slice knots via annulus twists}, \\
 arXiv:1305.7492v2 [math.GT], preprint.

\bibitem{Ak1}
S. Akbulut,
\emph{Knots and exotic smooth structures on $4$-manifolds},
J. Knot Theory Ramifications \textbf{2} (1993), no. 1, 1--10. 

\bibitem{Ak2}
S.~Akbulut,
\emph{On $2$-dimensional homology classes of $4$-manifolds},
Math. Proc. Cambridge Philos. Soc. \textbf{82} (1977), no. 1, 99--106.

\bibitem{AkbulutBook}
S.~Akbulut,
\emph{$4$-manifolds}, draft of a book (2012), \\ 
available at {\tt http://www.math.msu.edu/\~{}akbulut/papers/akbulut.lec.pdf}
 
\bibitem{Bo}
S. Boyer, 
\emph{Dehn surgery on knots}, Chapter 4 of the Handbook of
Geometric Topology, R.J. Daverman, R.B. Sher, ed., Amsterd
am, Elsevier, 2002.

\bibitem{Cerf}
J.~Cerf,
\emph{Sur les diffeomorphismes de la sphere de dimension trois $(\Gamma_{4}=0)$},
Lecture Notes in Mathematics, No. 53, Springer-Verlag, Berlin-New York (1968) xii+133 pp. 

\bibitem{GS}
F. Gonz\'alez-Acu\~{n}a and H. Short, 
\emph{Knot surgery and primeness}, 
Math. Proc. Cambridge Philos. Soc. \textbf{99} (1986), no. 1, 89--102.

\bibitem{Kirby}
R.~Kirby, 
\emph{Problems in Low-Dimensional Topology},
AMS/IP Stud. Adv. Math., {\bf 2}(2), Geometric topology (Athens, GA, 1993), 
35--473, Amer. Math. Soc., Providence, RI, 1997.

\bibitem{LO}
J.~Luecke and J.~Osoinach, 
\emph{Infinitely many knots admitting the same integer surgery}, 
arXiv:1407.1529. 
 
\bibitem{MO}
T.~Mrowka and P.~Ozsvath,
\emph{Low Dimensional Topology}.

\bibitem{Osoinach}
J.~Osoinach 
\emph{Manifolds obtained by surgery on an infinite number of knots in $S^3$},
Topology \textbf{45} (2006), 725--733. 


\bibitem{OS}
P.~Ozsvath and Z.~Szabo, 
\emph{The Dehn surgery characterization of the trefoil and the figure eight knot},
math.GT/0604079.

\bibitem{Rolfsen}
D. Rolfsen, 
\emph{Knots and Links}, 
Mathematics Lecture Series, No. 7. Publish or Perish, Inc., Berkeley, Calif., 1976.

\end{thebibliography}
\end{document}